\documentclass{article}
\usepackage{amssymb}
\usepackage{amsmath}
\usepackage{amsfonts}
\usepackage{graphicx}
\usepackage{graphics}

\newtheorem{theorem}{Theorem}

\newtheorem{corollary}[theorem]{Corollary}

\newtheorem{example}[theorem]{Example}

\newtheorem{lemma}[theorem]{Lemma}

\newenvironment{proof}[1][Proof]{\noindent\textbf{#1.} }{\ \rule{0.5em}{0.5em}}

\begin{document}

\title{Spectral results for the dominating induced matching problem}

\author{\textbf{Domingos M. Cardoso}, \textbf{Enide A. Martins} \\
{\small CIDMA-Center for Research and Development in Mathematics and Applications,}\\
{\small Department of Mathematics, University of Aveiro, 3810-193 Aveiro, Portugal}\\
{\small \texttt{dcardoso@ua.pt} \& \texttt{enide@ua.pt}} \and \textbf{Luis Medina}\\
{\small Departamento de Matem\'aticas, Universidad de Antofagasta, Antofagasta, Chile.}\\
{\small \texttt{luis.medina@uantof.cl}} \and \textbf{Oscar Rojo}\\
{\small Departamento de Matem\'{a}ticas, Universidad Cat\'{o}lica del Norte, Antofagasta, Chile. }\\
{\small \texttt{orojo@ucn.cl}}}

\maketitle

\begin{abstract}
A matching $M$ is a dominating induced matching of a graph, if every edge of the graph is either in $M$ or has
a common end-vertex with exactly one edge in $M$. The concept of complete dominating induced matching is introduced
as graphs where the vertex set can be partitioned into two subsets, one of them inducing an $1$-regular graph and
the other defining an independent set and such that all the remaining edges connect each vertex of one set to each vertex
of the other. The principal eigenvectors of the adjacency, Laplacian and signless Laplacian matrices of graphs with
complete dominating induced matchings are characterized and, therefore, the polynomial time recognition of graphs
with complete dominating induced matchings is stated. The adjacency, Laplacian and signless Laplacian spectrum of
graphs with complete dominating induced  matchings are characterized. Finally, several upper and lower bounds
on the cardinality of a dominating induced matching obtained from the eigenvalues of the adjacency, Laplacian and signless
Laplacian matrices are deduced and examples for which some of these bounds are tight are presented.
\end{abstract}

\bigskip


\bigskip


\baselineskip=0.30in

\section{Introduction}

Throughout this paper we consider undirected simple graphs $G$ of order $n>1$ with a vertex set $V(G)$ and edge
set $E(G)$. An element of  $E(G)$, which has the vertices $i$ and $j$ as end-vertices, is denoted by $ij$.
A matching $M$ of a graph $G$ is an edge subset such that there are no two edges in $M$ with a common end-vertex.
If $v \in V(G),$ then we denote the neighborhood of $v$ by $N_G(v)$, that is, $N_{G}(v) = \{w: vw \in E(G)\}.$
The number of neighbors of $v \in V(G)$ is denoted by $d(v)$ and called, as usually, degree of $v.$
If $G$ is such that $\forall v \in V(G) \;\; d(v)=p$ then we say that $G$ is $p$-regular. A graph induced by
a vertex subset $S \subset V(G)$ which is denoted by $G[S]$ is a subgraph of $G$ with vertex set $S$ and edge set
$E(G[S])=\{ij: i,j \in S \wedge ij \in E(G)\}$. An independent set of a graph $G$ is a vertex subset $S \subseteq V(G)$
which induces a $0$-regular graph $G[S]$.
A matching $M$ of $G$ is a \textit{dominating induced matching} (DIM) of $G$ if every edge of $G$ is either in
$M$ or has a common end-vertex with exactly one edge in $M$. A DIM is also called an
\textit{efficient edge domination set}. Observe that if $M$ is a DIM of $G$, then there
is a partition of $V(G)$ into two disjoint subsets $V(M)$ and $S$, where $S$ is an independent set. Conversely,
if there exists a graph $G$ such that its vertex set $V(G)$ can be partitioned into two vertex subsets $V_1$ and $V_2$,
where $V_1$ induces a matching and $V_2$ is an independent set, then the subset $M \subset E(G)$ of edges with both
ends in $V_1$ is a DIM. From now on, if $M$ is a DIM of $G$ then $V(G) \setminus V(M)$ will be denoted by $S$.

Not all graphs have a DIM, for instance the cycle with four vertices $C_4$ has no DIM. The \textit{DIM problem}
asks wether a given graph has a dominating induced matching. Dominating induced matchings have been studied for
instance in \cite{BHN10, BM11, BLR12, CCDS08, CL09, CKL11, K09}.

Now we introduce the concept of \textit{complete dominating induced matching}. Given a graph $G$, a DIM
$M \subset E(G)$ is complete if $E(G)=M \cup \{xy: x \in V(M), y \in V(G) \setminus V(M)\}$, that is, each vertex
of $V(M)$ is connected by an edge to each vertex of $V(G)\setminus V(M)$ which is not empty. As will see, the graphs
with complete dominating induced matchings are graphs with special spectral properties.

Throughout the paper we deal with adjacency, Laplacian and signless Laplacian matrices of graphs. The adjacency matrix of
a graph $G$ of order $n$ is the $n\times n$ symmetric matrix $A\left(G\right) =\left( a_{ij}\right)$ where $a_{ij}=1$ if $ij \in E(G)$
and $a_{ij}=0$ otherwise. The Laplacian (signless Laplacian) matrix of $G$ is the matrix $L(G)=D(G)-A(G)$ ($Q(G)=D(G)+A(G)$),
where $D(G)$ is the $n \times n$ diagonal matrix of vertex degrees of $G$.  The matrices $A\left(G\right)$, $L\left(G\right)$
and $Q\left(G\right)$ are all real and symmetric. From Ger\v{s}gorin's theorem, it follows that the eigenvalues of
$L\left(G\right)$ and $Q\left( G\right)$ are nonnegative real numbers. The spectrum of a matrix $C$ is denoted by $\sigma(C)$
and, in the particular cases of $A(G)$, $L(G)$ and $Q(G)$, their spectra are denoted by $\sigma _{A}(G)$, $\sigma _{L}(G)$ and
$\sigma _{Q}(G)$,  respectively. In this text, $\sigma _A(G) = \{\lambda_1^{[i_1]}, \ldots, \lambda_p^{[i_p]}\},$
$\sigma _L(G) = \{\mu_1^{[j_1]}, \ldots, \mu_q^{[j_q]}\}$ and $\sigma _{Q}(G)  = \{q_1^{[k_1]}, \ldots, q_r^{[k_r]}\}$)
means that $\lambda_s$, $\mu_s$ or $q_s$ is an adjacency, Laplacian or signless Laplacian eigenvalue with multiplicity
$i_s$, $j_s$ or $k_s$. As usually, we denote the eigenvalues of $A\left(G\right),$ $L\left(G\right)$ and $Q\left(G\right)$)
in non increasing order, that is, $\lambda_1(G) \ge \cdots \ge \lambda_n(G),$  $\mu_1(G) \ge \cdots \ge \mu_n(G)$
and $q_1(G) \ge \cdots \ge q_n(G)$. Considering a graph $G$, the largest eigenvalue of the adjacency $A(G)$, Laplacian $L(G)$
and signless Laplacian $Q(G)$ matrix of $G$ will be denoted, respectively, by $\rho(A(G))$, $\rho(L(G))$ and $\rho(Q(G))$.
The eigenvectors associated to $\rho(A(G))$, $\rho(L(G))$ or $\rho(Q(G))$ are called the principal eigenvectors of
$A(G)$, $L(G)$ or $Q(G)$, respectively. As usually, $\rho(A(G))$ is called the index of $G$ and it is also denoted $\rho(G)$.
For an arbitrary square matrix $C$ the $i$-th eigenvalue and its trace are denoted by $\lambda_i(C)$ and $\text{tr}(C)$,
respectively.

Throughout this paper, $\mathbf{j}_k$ denotes the all one vector with $k$ entries.

\section{Adjacency, Laplacian and signless Laplacian spectra of graphs with complete dominating induced matchings}

Given a graph $H$ of order $n$ with a complete dominating induced matching $M$ such that $|M|=m$, we may define $H$ using
the join operation as follows. Let $H_r=mK_2,$ with $r=2m$ and $H_s=G[V(G) \setminus V(M)]$, with $s=n-r$, a null graph
of order $s$ (that is, a graph formed by $s$ isolated vertices). Then $H=H_{r}\vee H_{s}$, that is, $H$ is the join of the
graphs $H_r$ and $H_s$.


Consider the two above vertex disjoint graphs $H_r$ and $H_s$ of orders $r$ and $s$, respectively. We label the vertices of
$H=H_r \vee H_s$, with the labels $1,2,\ldots ,r$ for the vertices of $H_r$ and with the labels $r+1, \ldots, r+s$, for the
vertices of $H_s$. Let $C\left(H\right)$ be a matrix on $H=H_r \vee H_s$. If $C\left(H\right)=L\left(H\right)$ or
$C\left(H\right)=A\left(H\right)$ or $C\left(H\right)=Q\left(H\right)$ then, using the above mentioned labeling for the vertices
of $H,$ we obtain
\begin{equation}\label{general_matrix}
C\left(H\right) = \left[\begin{array}{cc}
                         C_{1}                                  & \delta\mathbf{j}_{r}\mathbf{j}_{s}^{T} \\
                         \delta\mathbf{j}_{s}^{T}\mathbf{j}_{r} & C_2
                        \end{array}%
                   \right]
\end{equation}%
where $C_{1}=A(H_r)$ and $C_{2}=A(H_s)$ or $C_{1}=L(H_r)+sI_r$ and $C_{2}=L(H_s)+rI_s$ or $C_{1}=Q(H_r)+sI_r$ and $C_{2}=Q(H_s)+rI_s$,
when $C(H)$ is the adjacency, Laplacian or signless Laplacian matrix of $H$, respectively. In any case, in \eqref{general_matrix}
we have $\delta \ne 0$. Notice that
\begin{equation*}
C_{1}\mathbf{j}_{r}=\gamma _{1}\mathbf{j}_{r}\text{ and } C_{2}\mathbf{j}_{s}=\gamma_{2}\mathbf{j}_{s},
\end{equation*}
with $\gamma_1=1$ and $\gamma_2=0$ (when $C(H)$ is the adjacency matrix) or $\gamma_1=-s$ and $\gamma_2=-r$ (when $C(H)$ is the Laplacian
matrix) or $\gamma_1=2+s$ and $\gamma_2=r$ (when $C(H)$ is the signless Laplacian matrix).

Let us consider the matrix
\begin{equation}\label{matrix_B}
B = \left[\begin{array}{cc}
                 \gamma_1         & \delta \sqrt{rs}\\
                 \delta \sqrt{rs} & \gamma_2        \\
          \end{array}\right]
\end{equation}
and its eigenvalues
\begin{eqnarray}
\theta_1 &=& \frac{1}{2}\left(\gamma_1 + \gamma_2 + \sqrt{(\gamma_1 - \gamma_2)^2 + 4\delta^2rs}\right) \label{marca1}\\
\theta_2 &=& \frac{1}{2}\left(\gamma_1 + \gamma_2 - \sqrt{(\gamma_1 - \gamma_2)^2 + 4\delta^2rs}\right).\label{marca2}
\end{eqnarray}

\begin{lemma}\label{lemma1}
If $B$ is the matrix in \eqref{matrix_B} and $\theta \in \sigma(B)$ has an associated eigenvector
$\left[\begin{array}{c}
              1 \\
              x \\
       \end{array}\right]$, then $\theta \in \sigma(C(H))$ and $\left[\begin{array}{r}
                                                                      \textbf{j}_r \\
                                                                \sqrt{\frac{r}{s}}x\textbf{j}_s \\
                                                                      \end{array}\right]$ is an
associated eigenvector.
\end{lemma}

\begin{proof}
Since $\theta \in \sigma(B)$ and $\left[\begin{array}{c}
                                               1 \\
                                               x \\
                                        \end{array}\right]$ is an associated eigenvector, then
\begin{equation}
\left[\begin{array}{cc}
\gamma_{1}       & \delta \sqrt{rs} \\
\delta \sqrt{rs} & \gamma_{2}
\end{array}\right] \left[\begin{array}{c}
                                1 \\
                                x%
                         \end{array}%
                         \right] = \theta \left[\begin{array}{c}
                                                       1 \\
                                                       x
                                                \end{array}\right] .\label{xeq1}
\end{equation}
On the other hand,
\begin{eqnarray}
\left[\begin{array}{cc}
             C_{1}                            & \delta \mathbf{j}_{r}\mathbf{j}_{s}^{T} \\
      \delta \mathbf{j}_{s}\mathbf{j}_{r}^{T} & C_{2}
\end{array}\right] \left[\begin{array}{c}
                         \mathbf{j}_{r} \\
                   \sqrt{\frac{r}{s}}x\mathbf{j}_{s}%
                         \end{array}\right]
         &=& \left[\begin{array}{c}
             \gamma_{1}\mathbf{j}_{r} + \delta \sqrt{rs}x\mathbf{j}_{r} \\
             \delta r\mathbf{j}_{s}   + \sqrt{\frac{r}{s}}x\gamma_{2}\mathbf{j}_{s}
             \end{array}\right] \nonumber\\
         &=& \left[\begin{array}{c}
                   \left(\gamma_{1}+\delta\sqrt{rs}x\right)\mathbf{j}_{r} \\
                   \left(\delta \sqrt{rs}+x \gamma_{2}\right)\sqrt{\frac{r}{s}}\mathbf{j}_{s}
                   \end{array}\right] \nonumber \\
         &=&\theta \left[\begin{array}{c}
                         \mathbf{j}_{r} \\
                         \sqrt{\frac{r}{s}}x\mathbf{j}_{s}%
                         \end{array}\right]. \label{marca3}
\end{eqnarray}
Notice that the equality \eqref{marca3} is obtained as a consequence of \eqref{xeq1}. Hence,
$\theta \in \sigma \left( C\left( H \right) \right)$ and $\left[\begin{array}{c}
                                                          \mathbf{j}_{r}\\
                                                          \sqrt{\frac{r}{s}}x\mathbf{j}_{s}%
                                                                \end{array}\right]$ is an
associated eigenvector.
\end{proof}

The next theorem allows the polynomial time recognition of graphs with complete DIMs.

\begin{theorem}\label{cor_dim_eigenvectors}
Let $H$ be a graph of order $n$ with a complete DIM, $M \subset E(H)$ such that $|M|=m$. Then, the spectral radius of
the adjacency, Laplacian and signless Laplacian matrix of $H$, as well as the corresponding principal eigenvectors,
are the following:
\begin{enumerate}
\item $\rho(A(H)) = \frac{1 + \sqrt{1 + 8m(n-2m)}}{2}$ and the corresponding eigenvector is
      \begin{equation*}
      \mathbf{u} = \left[\begin{array}{c}
                         \textbf{j}_{2m}\\
                         \frac{\rho(A(H)) - 1}{n-2m}\textbf{j}_{n-2m}\\
                         \end{array}\right].
      \end{equation*}\label{cor_adj}
\item $\rho(L(H)) = n$ and the corresponding eigenvector is
      \begin{equation*}
      \mathbf{v} = \left[\begin{array}{c}
                         \textbf{j}_{2m}\\
                         -\frac{2m}{n-2m}\textbf{j}_{n-2m}\\
                         \end{array}\right].
      \end{equation*}\label{cor_lap}
\item $\rho(Q(H))=\frac{2+n+\sqrt{(2+n)^2-16m}}{2}$ and the corresponding eigenvector is
      \begin{equation*}
      \mathbf{w} = \left[\begin{array}{c}
                       \mathbf{j}_{2m} \\
                       \frac{\rho(Q(H))-(n-2m+2)}{n-2m}\mathbf{j}_{n-2m}%
                       \end{array}\right].
      \end{equation*}\label{cor_slap}
\end{enumerate}
\end{theorem}

\begin{proof}
As in  Lemma~\ref{lemma1}, consider $H=H_r \vee H_s$ with $H_r=mK_2$ and $H_s=H[S]$, where $S$ is an independent set of $H$ of size $n-2m$.
Then $r=2m$, $s=n-2m$ and we may analyze each of the following cases.
\begin{enumerate}
\item Assuming $\delta=1$, $C_1=A(H_{2m})$ and $C_2=A(H_{n-2m})$, then the matrix \eqref{matrix_B} becomes
      $$
      B = \left[\begin{array}{cc}
                 1               & \sqrt{2m(n-2m)}\\
                 \sqrt{2m(n-2m)} & 0               \\
          \end{array}\right].
      $$
      Therefore, according to \eqref{marca1}-\eqref{marca2},
      $$
      \sigma(B)=\{\frac{1 + \sqrt{1 + 8m(n-2m)}}{2}, \frac{1 - \sqrt{1 + 8m(n-2m)}}{2}\}.
      $$
      Let $\rho = \frac{1 + \sqrt{1 + 8m(n-2m)}}{2}$ and assume that $\mathbf{x} = \left[\begin{array}{c}
                                                                                                1 \\
                                                                                                x \\
                                                                                         \end{array}\right]$
      is an eigenvector of $B$ associated to $\rho$. From the eigenvalue equation
      $$
      B\mathbf{x} = \left[\begin{array}{cc}
                            1               & \sqrt{2m(n-2m)}\\
                            \sqrt{2m(n-2m)} & 0               \\
                           \end{array}\right]\left[\begin{array}{c}
                                                          1 \\
                                                          x \\
                                                   \end{array}\right] = \rho\left[\begin{array}{c}
                                                                                            1 \\
                                                                                            x \\
                                                                                     \end{array}\right],
      $$
      it follows that $1+x\sqrt{2m(n-2m)} = \rho \Leftrightarrow x = \frac{\rho-1}{\sqrt{2m(n-2m)}}$. Therefore, using Lemma~\ref{lemma1},
      $$
      \mathbf{u} = \left[\begin{array}{r}
                         \textbf{j}_{2m}\\
                         \sqrt{\frac{2m}{n-2m}}x\textbf{j}_{n-2m}\\
                         \end{array}\right] = \left[\begin{array}{r}
                                                    \textbf{j}_{2m}\\
                                                    \frac{\rho - 1}{n-2m}\textbf{j}_{n-2m}\\
                                                    \end{array}\right].
      $$
      is an eigenvector of $A(H)$ associated to the eigenvalue $\rho$. Since $\mathbf{u}>0$ and since $A(H)$ is an irreducible nonnegative
      matrix, we may conclude that $\rho=\rho(A(H))$, that is, $\rho$ is the spectral radius of $A(H)$.

\item Assuming $\delta=-1$,  $C_1=L(H_{2m})+(n-2m)I_{2m}$ and $C_2=L(H_{n-2m})+2mI_{n-2m}$, then the matrix $B$ in \eqref{matrix_B} becomes
      $$
      B = \left[\begin{array}{cc}
                 n-2m             & -\sqrt{2m(n-2m)}\\
                 -\sqrt{2m(n-2m)} & 2m               \\
          \end{array}\right].
      $$
      Therefore, according to \eqref{marca1}-\eqref{marca2}, $\sigma(B)=\{n, 0\}$. Let $\rho=n$ and let us assume that
      $\mathbf{x} = \left[\begin{array}{c}
                                 1 \\
                                 x \\
                          \end{array}\right]$ is an eigenvector of $B$ associated to $\rho=n$. From the eigenvalue equation
      $$
      B\mathbf{x} = \left[\begin{array}{cc}
                            n-2m            & -\sqrt{2m(n-2m)}\\
                            -\sqrt{2m(n-2m)} & 2m               \\
                           \end{array}\right]\left[\begin{array}{c}
                                                          1 \\
                                                          x \\
                                                   \end{array}\right] = n \left[\begin{array}{c}
                                                                                       1 \\
                                                                                       x \\
                                                                                \end{array}\right],
      $$
      it follows that $n - 2m - x\sqrt{2m(n-2m)} = n \Leftrightarrow x = \frac{-2m}{\sqrt{2m(n-2m)}}$. Therefore, using Lemma~\ref{lemma1},
      $$
      \mathbf{v} = \left[\begin{array}{c}
                         \textbf{j}_{2m}\\
                         \sqrt{\frac{2m}{n-2m}}x\textbf{j}_{n-2m}\\
                         \end{array}\right] = \left[\begin{array}{c}
                                                    \textbf{j}_{2m}\\
                                                    \frac{-2m}{n-2m}\textbf{j}_{n-2m}\\
                                                    \end{array}\right].
      $$
      is an eigenvector of $L(H)$ associated to $\rho$. Since the spectral radius of the Laplacian matrix of any graph does not exceed the order
      of the graph (see, for instance, \cite[Prop.7.1.1]{CRS10}), we may conclude that $\rho=\rho(L(H))$, that is, $\rho=n$ is the spectral radius
      of $L(H)$.
\item Assuming $\delta=1$, $C_1=Q(H_{2m})+(n-2m)I_{2m}$ and $C_2=Q(H_{n-2m})+2mI_{n-2m}$, then the matrix \eqref{matrix_B} becomes
      $$
      B = \left[\begin{array}{cc}
                 n - 2m + 2      & \sqrt{2m(n-2m)}\\
                 \sqrt{2m(n-2m)} & 2m               \\
          \end{array}\right].
      $$
      Therefore, according to \eqref{marca1}-\eqref{marca2}, $\sigma(B)=\{\frac{n+2+\sqrt{(n+2)^2-16m}}{2},\frac{n+2-\sqrt{(n+2)^2-16m}}{2}\}$.
      Let $\rho=\frac{n+2+\sqrt{(n+2)^2-16m}}{2}$ and let us assume that $\mathbf{x} = \left[\begin{array}{c}
                                                                                                    1 \\
                                                                                                    x \\
                                                                                             \end{array}\right]$ is an eigenvector of $B$ associated
      to $\rho$. From the eigenvalue equation
      $$
      B\mathbf{x} = \left[\begin{array}{cc}
                            n-2m+2          & \sqrt{2m(n-2m)}\\
                            \sqrt{2m(n-2m)} & 2m               \\
                           \end{array}\right]\left[\begin{array}{c}
                                                          1 \\
                                                          x \\
                                                   \end{array}\right] = \rho \left[\begin{array}{c}
                                                                                       1 \\
                                                                                       x \\
                                                                                \end{array}\right],
      $$
      it follows that $n - 2m +2 + x\sqrt{2m(n-2m)} = \rho \Leftrightarrow x = \frac{\rho-(n-2m+2)}{\sqrt{2m(n-2m)}}$. Therefore,
      using Lemma~\ref{lemma1},
      $$
      \mathbf{w} = \left[\begin{array}{c}
                         \textbf{j}_{2m}\\
                         \sqrt{\frac{2m}{n-2m}}\frac{\rho-(n-2m+2)}{\sqrt{2m(n-2m)}}\textbf{j}_{n-2m}\\
                         \end{array}\right] = \left[\begin{array}{c}
                                                    \textbf{j}_{2m}\\
                                                    \frac{\rho - (n-2m+2)}{n-2m}\textbf{j}_{n-2m}\\
                                                    \end{array}\right].
      $$
      is an eigenvector of $Q(H)$ associated to $\rho$. Since $\rho$ is the spectral radius of the real symmetric matrix $B$, it follows
      that $\rho > n - 2m + 2$, and then $\mathbf{w}>0$. From this and from the fact that $Q(G)$ is an irreducible nonnegative
      symmetric matrix, we may conclude that $\rho=\rho(Q(H))$, that is, $\rho$ is the spectral radius of $Q(H)$.
\end{enumerate}
\end{proof}

Applying Theorem~\ref{cor_dim_eigenvectors}, we may recognize in polynomial time if a graph has or not a complete DIM.
Notice that from the principal eigenvectors of the adjacency, Laplacian or signless Laplacian matrix it is possible to
identify the vertices belonging to the complete DIM (if there exists). Furthermore, since when an edge is deleted the
spectral radius decreases, we may conclude the following corollary.

\begin{corollary}
Let $G$ be a graph of order $n$ with a DIM, $M \subset E(G)$ such that $|M|=m$. Then the spectral radius of the adjacency and
signless Laplacian matrix of $G$ has the following upper bounds.
\begin{enumerate}
\item $\rho(A(G)) \le \frac{1 + \sqrt{1 + 8m(n-2m)}}{2}$;
\item $\rho(Q(G)) \le \frac{2+n+\sqrt{(2+n)^2-16m}}{2}$.
\end{enumerate}
\end{corollary}

In this corollary, the Laplacian case is not considered, since for any graph the largest Laplacian eigenvalue is not greater than
the order of the graph.

\begin{theorem}\label{th_dim_eigenvalues}
Let $H$ be a graph of order $n$ with a complete DIM, $M \subset E(H)$ such that $|M|=m$. Then, the adjacency,
Laplacian and signless Lapacian spectra of $H$ are given by:
\begin{enumerate}
\item $\sigma_A(H)=\{\frac{1 \pm \sqrt{1+8m(n-2m)}}{2}, 1^{[m-1]}, 0^{[n-2m-1]}, (-1)^{[m]}\}$.\label{th_dim_1}
\item $\sigma_L(H)=\{n, (n-2m+2)^{[m]}, (n-2m)^{[m-1]}, (2m)^{[n-2m-1]}, 0\}$.\label{th_dim_2}
\item $\sigma_Q(H)=\{\frac{2+n \pm \sqrt{(2+n)^2-16m}}{2}, (n-2m+2)^{[m-1]}, (n-2m)^{[m]}, (2m)^{[n-2m-1]}\}$.\label{th_dim_3}
\end{enumerate}
\end{theorem}

\begin{proof}
Taking into account that $H=H_r \vee H_s$, with $H_r=mK_2$ and $H_s=G[V(G) \setminus V(M)],$ where $r=2m$ and $s=n-2m$, we may apply the results
obtained in \cite{cardoso_et_al2013} as follows.
\begin{enumerate}
\item \textbf{The adjacency spectrum}: Applying Theorem 5 in \cite{cardoso_et_al2013}, it follows that
      $\sigma_A(H_r \vee H_s) = \sigma_A(H_r) \setminus \{1\} \cup \sigma_A(H_s)\setminus\{0\} \cup \sigma(\widetilde{C}),$ where
      $$\widetilde{C}=\left[\begin{array}{cc}
                            1         & \sqrt{rs} \\
                            \sqrt{rs} & 0%
              \end{array}
          \right].$$
      Therefore, $\sigma_A(H_r \vee H_s)=\{\frac{1+\sqrt{1+4rs}}{2}, 1^{[m-1]}, 0^{[n-2m-1]}, (-1)^{[m]}, \frac{1-\sqrt{1+4rs}}{2}\}$.
\item \textbf{The Laplacian spectrum}: Applying Theorem 8 in \cite{cardoso_et_al2013}, it follows that
      $\sigma_L(H_r \vee H_s)= (s+\sigma_L(H_r)\setminus \{0\}) \cup (r + \sigma_L(H_s)\setminus\{0\}) \cup \sigma(\widetilde{C}),$ where
      $$\widetilde{C}=\left[\begin{array}{cc}
                  s          & -\sqrt{rs} \\
                  -\sqrt{rs} & r%
                  \end{array}
                  \right].$$
      Therefore, $\sigma_L(H_r \vee H_s)=\{r+s, (s+2)^{[m]}, s^{[m-1]}, r^{[n-2m-1]}, 0\}$.
\item \textbf{The signless Laplacian spectrum}: Applying Theorem 3 in \cite{cardoso_et_al2013}, it follows that
      $\sigma_Q(H_r \vee H_s)= (s+\sigma_Q(H_r) \setminus \{2\}) \cup (r+\sigma_Q(H_s)\setminus\{0\}) \cup \sigma(\widetilde{C}),$ where
      $$\widetilde{C}=\left[\begin{array}{cc}
                             2+s         & \sqrt{rs} \\
                             \sqrt{rs}   & r%
              \end{array}
          \right].$$
      Therefore, $\sigma_Q(H_r \vee H_s)=\{\frac{2+r+s \pm \sqrt{(2+r+s)^2-8r}}{2}, (s+2)^{[m-1]}, s^{[m]}, r^{[n-2m-1]}\}$.
\end{enumerate}
\end{proof}

\begin{example}
let $H$ be a graph obtained from the graph depicted in Figure~\ref{figure1} adding the edges $18, 19, 29, 37, 39, 47, 49, 57, 67, 68$.
Then $H$ is a graph with a complete DIM, $M=\{12, 34, 56\}$. Applying Theorem~\ref{th_dim_eigenvalues}, we obtain
\begin{enumerate}
\item $\sigma_A(H) = \{4.772, 1^{[2]}, 0^{[2]}, -1^{[3]}, -3.772\}$;
\item $\sigma_L(H) = \{9, 6^{[2]}, 5^{[3]}, 3^{[2]}, 0\}$;
\item $\sigma_Q(H) = \{9.772, 6^{[2]}, 5^{[2]}, 3^{[3]}, 1.228\}$.
\end{enumerate}
Applying Theorem~\ref{cor_dim_eigenvectors}, we may conclude that the principal eigenvectors of $A(H)$, $L(H)$ and $Q(H)$ is
$\mathbf{u} = \left[\begin{array}{c}
                         \textbf{j}_{6}\\
                         1.2573\textbf{j}_{3}\\
                         \end{array}\right],$ $\mathbf{v} = \left[\begin{array}{c}
                                                                  \textbf{j}_{6}\\
                                                                  -2\textbf{j}_{3}\\
                                                                  \end{array}\right]$ and $\mathbf{w} = \left[\begin{array}{c}
                                                                                                              \textbf{j}_{6}\\
                                                                                                        1.5907\textbf{j}_{3}\\
                                                                                                        \end{array}\right]$, respectively.
\end{example}

\section{Lower and upper bounds on the size of a DIM, obtained from the adjacency, Laplacian and signless Laplacian spectra}

From now on, we denote a graph with a complete DIM, $M$, by $K_{M,S}$, where $M$ is a dominating induced matching,
$S$ is an independent set and each vertex of $S$ is connected by an edge to each vertex of $V(M)$.

\subsection{Bounds obtained from the adjacency spectra of graphs with a DIM}

\begin{lemma}
Let $G$ be a graph of order $n$ with a DIM $M \subseteq E(G)$. Then
\begin{equation}\label{inequality_1}
\left(\frac{n}{2}\right)^2 \ge \rho(\rho-1),
\end{equation}
where $\rho=\rho(G)$, with equality if and only if $n=4m$ and $G=K_{M,S}$.
\end{lemma}

\begin{proof}
Assuming that $|M|=m$ and $S=V(G) \setminus S$, since $f(x)=4x^2-4x$ is a strictly increasing function for $x>1/2$,
it follows that
\begin{equation}\label{inequality_2}
4(\rho^2-\rho) \le 4(\rho^2(K_{M,S})-\rho(K_{M,S})) = 8mn - 16m^2.
\end{equation}
Moreover, $(n-4m)^2 \ge 0 \Leftrightarrow 8mn - 16m^2 \le n^2$. Therefore, from
\eqref{inequality_2}, the inequality \eqref{inequality_1} follows. It is immediate that \eqref{inequality_1} holds
as equality if and only if $4m=n$ and
$G=K_{M,S}$
\end{proof}

From the above lemma, taking into account the item \ref{th_dim_1} of Theorem~\ref{th_dim_eigenvalues}, we are able
to obtain the following result.

\begin{theorem}\label{th_u_and_l_bound}
Let $G$ be a graph of order $n$ with a DIM $M \subset E(M)$ such that $|M|=m$ and let $\rho=\rho(G)$.
If $G \ne K_{M,S}$, then
$$
\lceil \frac{1}{4}\left(n-\sqrt{n^2-4(\rho^2-\rho)}\right)\rceil \le m \le \lfloor \frac{1}{4}\left(n+\sqrt{n^2-4(\rho^2-\rho)}\right)\rfloor.
$$
\end{theorem}

\begin{proof}
Since from Theorem~\ref{th_dim_eigenvalues}-\ref{th_dim_1}, $\rho(K_{M,S})=\frac{1}{2}(1+\sqrt{1+8m(n-2m)})$,
then $\rho(G) \le \frac{1}{2}(1+\sqrt{1+8m(n-2m)})$ and, setting $\rho=\rho(G)$, after some algebraic
steps we get
\begin{equation}
4m^2 - 2nm + \rho^2 - \rho \le 0.\label{cor_2_inequality}
\end{equation}
Let $q(m)=4m^2 - 2nm + \rho^2 - \rho$. Since $G \ne K_{M,S}$, then $n^2-4(\rho^2-\rho)>0$ and
therefore $q(m)=0$ has two real roots
\begin{eqnarray*}
m_1 &=& \frac{1}{8}\left(2n - \sqrt{4n^2 - 16(\rho^2-\rho)}\right)\\
m_2 &=& \frac{1}{8}\left(2n - \sqrt{4n^2 - 16(\rho^2-\rho)}\right).
\end{eqnarray*}
Hence, the inequality \eqref{cor_2_inequality} holds when $m_1 \le m \le m_2$.
\end{proof}

\begin{example}\label{ex_1}
Let us consider the graph $G$ depicted in Figure~\ref{figure1} which has order $n=9$ and minimum degree $\delta(G)=2$.
\begin{figure}[!h]
\begin{center}
\includegraphics[width=4cm]{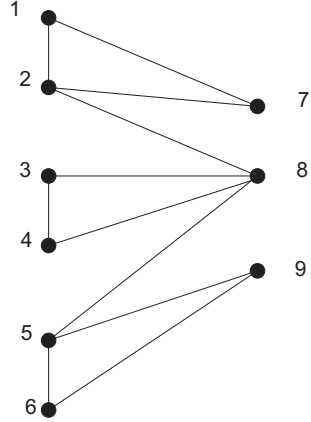}\\
\caption{A graph $G$ with a dominating induced matching $M=\{12, 34, 56\}$.}\label{figure1}
\end{center}
\end{figure}
Since the index of $G$ is $\rho(G)=2.6364$, then
$m_1=\frac{1}{4}\left(9-\sqrt{9^2-4(2.6364^2-2.6364)}\right)=0.254017$ and
$m_2=\frac{1}{4}\left(9+\sqrt{9^2-4(2.6364^2-2.6364)}\right)=4.245983$. Therefore, according to
Theorem~\ref{th_u_and_l_bound}, $1 \le m \le 4$.
\end{example}

Let $M$ be a DIM of $G$. Considering $|M|=m$ and labeling the
vertices of $V_1$ as $1, \ldots, 2m$, the adjacency matrix of $G$ is as follows:
\begin{eqnarray}
A\left( G\right) &=& \left[\begin{array}{cc}
                                P & R \\
                                R^{T} & 0%
                         \end{array}%
                   \right]. \label{adjacency_matrix}
\end{eqnarray}
One can see that $P^{2}=I$.

\begin{theorem}
Let $G$ be a graph of order $n,$ minimum degree $\delta=\delta(G)$, and index $\rho=\rho(G)$. If $G$ has
a dominating induced matching $M \subset E(G)$, then
\begin{equation}
|M| \ge \lceil \frac{n\left( 2\delta -\rho\right) }{2\left( 2\delta -1\right)} \rceil. \label{lb_on_dim}
\end{equation}
\end{theorem}

\begin{proof}
Let us assume that the graph $G$ has a dominating induced matching $M \subset E(G)$ such that $|M|=m$
and thus $V(G)$ can be partitioned into the vertex subsets $V(M)$ and $S$, where $S$ is an independent
set. Then the adjacency matrix of $G$ can be written as in \eqref{adjacency_matrix} and it follows:
\begin{eqnarray*}
\rho & \ge &\frac{1}{\sqrt{n}}[\begin{array}{cc}
                                       \mathbf{j}_{2m}^{T} & \mathbf{j}_{n-2m}^{T}
                                       \end{array}\mathbf{]}\left[\begin{array}{cc}
                                                                         P     & R \\
                                                                         R^{T} & 0%
                                                                  \end{array}%
                                                            \right] \frac{1}{\sqrt{n}}\left[\begin{array}{c}
                                                                                            \mathbf{j}_{2m} \\
                                                                                            \mathbf{j}_{n-2m}
                                                                                            \end{array}%
                                                                                            \right]\\
             &  =  &\frac{1}{n}\left[\begin{array}{cc}
                                    \mathbf{j}_{2m}^{T}P+\mathbf{j}_{n-2m}^{T}R^{T} & \mathbf{j}_{2m}^{T}R
                                    \end{array}\right] \left[\begin{array}{c}
                                                             \mathbf{j}_{2m} \\
                                                             \mathbf{j}_{n-2m}%
                                                             \end{array}\right] \\
             &  =  &\frac{1}{n}\left(\mathbf{j}_{2m}^{T}P\mathbf{j}_{2m}+
                              \mathbf{j}_{n-2m}^{T}R^{T}\mathbf{j}_{2m}+\mathbf{j}_{2m}^{T}R\mathbf{j}_{n-2m}\right)\\
             &  =  &\frac{1}{n}\left( \mathbf{j}_{2m}^{T}P\mathbf{j}_{2m}+2\mathbf{j}_{n-2m}^{T}R^{T}\mathbf{j}_{2m}\right)\\
             &  =  &\frac{1}{n}\left( 2m+2\sum_{v \in S}d\left( v\right) \right)\\
             & \ge &\frac{1}{n}\left( 2m+2\left( n-2m\right) \delta \right).
\end{eqnarray*}
Therefore,
\begin{equation*}
\frac{n\left(\rho-2\delta \right) }{2\left(1-2\delta \right)}=
\frac{n\left(2\delta-\rho\right)}{2\left(2\delta-1\right)} \le m.
\end{equation*}
\end{proof}

\begin{example}\label{ex_2}
The graph $G$ depicted in Figure~\ref{figure1} is an example for which the lower bound \eqref{lb_on_dim} is tight.
In fact, since the graph $G$ has a dominating induced matching $M \subset E(G)$, and $n=9$, $\delta(G)=2$ and
$\rho(G)=2.6364$, it follows that
$\frac{n\left(2\delta(G) -\rho(G)\right) }{2\left(2\delta(G) -1\right) }= 2.0454$ and, therefore,
$\lceil \frac{n\left( 2\delta(G) - \rho(G)\right) }{2\left( 2\delta(G) -1\right)}\rceil =3 = |M|$.
\end{example}

Before introducing the next result, let us recall the following classical Cauchy interlacing theorem.

\begin{theorem}[Cauchy interlacing theorem \cite{CRS10}]\label{CauchyInterlacingTh}
Let
\[
A=\left[\begin{array}{cc}
         B & C^{\ast } \\
         C & D%
        \end{array}%
  \right]
\]%
be a $p \times p$ Hermitian matrix and $B$ a $q \times q$ matrix with $q<p.$ Then
\[
\lambda _{k}\left( A\right) \ge \lambda _{k}\left( B\right) \ge \lambda_{k+p-q}\left( A\right) \quad \hbox{for } k=1,2, \ldots, q.
\]
\end{theorem}

Now, applying this theorem to the adjacency matrix of a graph $G$ with an induced matching, we may conclude the following result.

\begin{theorem}
Let $G$ be a graph and let $M \subseteq E(G)$ be an induced matching such that $\left\vert M \right\vert = m.$
Then $\sigma(A(G))$ includes $m$ eigenvalues not greater than $-1$ and $m$ eigenvalues not less than $1.$
\end{theorem}

\begin{proof}
Let $A\left( G\right) =\left[\begin{array}{cc}
                                        P & R \\
                                        R^{T} & 0%
                                 \end{array}\right].$ Since the order of $P$ is $2m$ and the order of $A\left( G\right) $ is $n,$
then $R$ is a $2m\times \left( n-2m\right)$ matrix. Applying Theorem~\ref{CauchyInterlacingTh} to $A\left(G\right)$, setting $k=m$
and $k=m+1,$ respectively, we obtain%
\begin{eqnarray*}
\lambda _{m}\left( A\left( G\right) \right) &\ge &\lambda_{m}\left(P\right) = 1 \\
                                         -1 & =  &\lambda_{m+1}\left( P\right)\\
                                            &\ge & \lambda_{m+1+n-2m}\left( A\left(G\right) \right) =\lambda_{n-m+1}\left(A\left( G\right)\right).
\end{eqnarray*}%
Then $A\left( G\right)$ has $m$ eigenvalues not greater than $-1$ and $m$ eigenvalues not less than $1.$
\end{proof}

As immediate consequence we have the following corollary.

\begin{corollary}\label{cor_1}
Let $G$ be a graph with spectrum (multiset of adjacency eigenvalues) $\sigma_A(G)$,
$\Lambda^{-}=\left\{ \lambda \in \sigma_A(G): \lambda \le -1\right\}$
and $\Lambda^{+}=\left\{ \lambda \in \sigma_AG): \lambda \ge 1\right\}.$ If $M \subseteq E(G)$ is an induced matching of
$G$, then
\begin{equation}
\left\vert M \right \vert \le \min \left\{ \left\vert \Lambda^{-}\right\vert,\left\vert \Lambda^{+}\right\vert \right\}.\label{im_up_bound}
\end{equation}
\end{corollary}

\begin{example}\label{ex_3}
Considering the graph $G$ of the Example~\ref{ex_1} and taking into account that
$\sigma_A(G)=\{-2.0664, -1^{[4]},-0.2222, 1.6522, 2, 2.6364\}$, it follows that
$\Lambda^{-}=\{-2.0664,-1^{[4]}\}$ and $\Lambda^{+}=\{1.6522, 2, 2.6364\}$. Therefore,
according to Corollary~\ref{cor_1}, if $M \subset E(G)$ is an induced matching, then
$$
|M| \le 3.
$$
In this case, if $M$ is a dominating induced matching, combining \eqref{lb_on_dim} with \eqref{im_up_bound} we may conclude
that $|M|=3$.
\end{example}

\begin{example}\label{ex_3}
The graph $G$ depicted in Figure~\ref{figure2} is an example for which the upper bound \eqref{im_up_bound} is tight.
\begin{figure}[!h]
\begin{center}
\includegraphics[width=4cm]{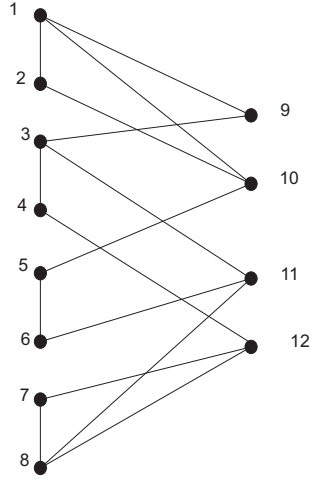}\\
\caption{A graph $G$ with a dominating induced matching $M=\{12, 34, 56, 78\}$.}\label{figure2}
\end{center}
\end{figure}
In fact, since the graph $G$ has an induced matching $M \subset E(G)$, and its spectrum $\sigma(G)$ is equal to
{\footnotesize $$
\{-2.156, -1.870,-1.597, -1.311, -0.897,-0.547, 0.034, 0.579, 1.386, 1.481, 2.308, 2.590\},
$$}
\noindent it follows that $|\Lambda^{-}|=4$ and $|\Lambda^{+}|=4$.
Therefore, $|M| \le 4$.
\end{example}

\subsection{Lower bounds obtained from the Laplacian spectra of graphs with a DIM}

Let us consider a graph $G$ of order $n$, with a DIM, $M \subset E(G)$, such that $|M|=m$ and thus $V(M)$ and
the independent set $S$ is a partition of $V(G)$. Let $D_1$ and $D_2$ be the diagonal matrices whose diagonal
entries are the degrees of the vertices in $V(M)$ and $S$, respectively. The Laplacian matrix of $G$ can be
written as
$$
L(G) = \left[\begin{array}{cc}
                    D_1 - P & -R \\
                    - R^T   &  D_2 \\
             \end{array}
       \right].
$$
Then, considering the largest eigenvalue $\mu_1$ of $L(G)$, we have
\begin{eqnarray*}
\mu_1 & \ge & \frac{1}{\sqrt{n}}\left[\textbf{j}^T_{2m} \; -\textbf{j}^T_{n-2m}\right]\left[\begin{array}{cc}
                                                                                 D_1 - P & -R \\
                                                                                  -R^T   &  D_2 \\
                                                                          \end{array}\right]\frac{1}{\sqrt{n}}\left[\begin{array}{c}
                                                                                                                    \textbf{j}_{2m} \\
                                                                                                                    -\textbf{j}_{n-2m} \\
                                                                                                             \end{array}\right]\\
      & = & \frac{1}{n}\left[\textbf{j}^T_{2m}D_1-\textbf{j}^T_{2m}P+\textbf{j}^T_{n-2m}R^T-\textbf{j}^T_{2m}R-\textbf{j}^T_{n-2m}D_2\right]\left[\begin{array}{c}
                                                                                                                                                  \textbf{j}_{2m}\\
                                                                                                                                                  -\textbf{j}_{n-2m}\\                                                                                                                                                      \end{array}\right]\ \\
      & = & \frac{1}{n} \left(\textbf{j}^T_{2m}D_1\textbf{j}_{2m}-\textbf{j}^T_{2m}P\textbf{j}_{2m}+\textbf{j}^T_{n-2m}R^T\textbf{j}^T_{2m}+\textbf{j}^T_{2m}R\textbf{j}^T_{n-2m}+\textbf{j}^T_{n-2m}D_2\textbf{j}^T_{n-2m}\right)\\
      & = & \frac{1}{n}\left(\textbf{j}^T_{2m}D_1\textbf{j}^T_{2m}+\textbf{j}^T_{n-2m}D_2\textbf{j}^T_{n-2m}-\textbf{j}^T_{2m}P\textbf{j}^T_{2m}+2\textbf{j}^T_{n-2m}R^T\textbf{j}^T_{2m}\right)\\
      & = & \frac{1}{n}\left(\text{tr}(L(G))-2m+2\sum_{v \in S}{d(v)}\right)\\
      &\ge& \frac{1}{n}\left(\text{tr}(L(G))-2m+2(n-2m)\delta(G)\right).
\end{eqnarray*}

Therefore, $m \ge \frac{\text{tr}(L(G))-n(\mu_1-2\delta(G))}{2(2\delta(G)+1)}$.

\subsection{Lower bounds obtained from the signless Laplacian spectra of graphs with a DIM}

Let us consider a graph $G$ of order $n$, with a DIM, $M \subset E(G)$, such that $|M|=m$. As in the previous subsection, let $D_1$ and $D_2$
be the diagonal matrices whose diagonal entries are the degrees of the vertices in $V(M)$ and $S=V(G)\setminus V(M)$,
respectively. Then, the signless Laplacian matrix of $G$ can be written as follows.
$$
Q(G) = \left[\begin{array}{cc}
                    D_1 + P & R \\
                    R^T     &  D_2 \\
             \end{array}\right] = \left[\begin{array}{cc}
                                         D_1 & 0 \\
                                         0   & D_2 \\
                                        \end{array}\right] + \left[\begin{array}{cc}
                                                                    P     & R \\
                                                                    R^T   & 0 \\
                                                                   \end{array}\right].
$$
Let $q_1$ be the largest eigenvalue of $Q(G)$. Then we have
\begin{eqnarray*}
q_1 & \ge & \frac{1}{\sqrt{n}}\left[\textbf{j}^T_{2m} \; \textbf{j}^T_{n-2m}\right]Q(G)\frac{1}{\sqrt{n}}\left[\begin{array}{c}
                                                                                                  \textbf{j}_{2m} \\
                                                                                                  \textbf{j}_{n-2m} \\
                                                                                                  \end{array}\right]\\
    &  =  & \frac{1}{n}\left(\sum_{v \in V(G)}{d(v)}+\textbf{j}^T_{2m}P\textbf{j}_{2m}+\textbf{j}^T_{n-2m}R^T\textbf{j}^T_{2m}+\textbf{j}^T_{2m}R\textbf{j}^T_{n-2m}\right)\\
    &  =  & \frac{1}{n}\left(\sum_{v \in V(G)}{d(v)}+\textbf{j}^T_{2m}P\textbf{j}_{2m}+2\textbf{j}^T_{2m}R\textbf{j}^T_{n-2m}\right).
\end{eqnarray*}
Therefore,
\begin{eqnarray*}
q_1 & \ge & \frac{1}{n}\left(\text{tr}(Q(G))+2m + 2 \sum_{v \in S}{d(v)}\right)\\
    & \ge & \frac{1}{n}\left(\text{tr}(Q(G))+2m+2(n-2m)\delta(G)\right)
\end{eqnarray*}
and
\begin{equation}
\frac{n(q_1-2\delta(G))-\text{tr(Q(G))}}{2(1-2\delta(G))}=\frac{\text{tr}(Q(G))-n(q_1-2\delta(G))}{2(2\delta(G)-1)} \le m.\label{q-index-lb}
\end{equation}

\begin{example}
Considering the graph of Figure~\ref{figure1}, since
$$
\sigma_Q(G)=\{1.0000^{[4]}, 1.5858, 2.2679,4.0000, 4.4142, 5.7321\},
$$
the upper bound \eqref{q-index-lb} produces
$$
\frac{\text{tr}(Q(G))-n(q_1-2\delta(G))}{2(2\delta(G)-1)}=\frac{22-9(5.7321-4)}{2(4-1)} = 1.0685
$$
and thus $m \ge 2$.
\end{example}

{\bf Acknowledgements}.
Domingos M. Cardoso and Enide A. Martins partially supported by {\it FEDER} funds through {\it COMPETE}--Operational
Programme Factors of Competitiveness and by Portuguese funds through the {\it Center for Research and
Development in Mathematics and Applications} (University of Aveiro) and the Portuguese Foundation for Science
and Technology (``FCT--Funda\c{c}\~{a}o para a Ci\^{e}ncia e a Tecnologia''), within project PEst-C/MAT/UI4106/2011
with COMPETE number FCOMP-01-0124-FEDER-022690 and also to the Project PTDC/MAT/112276/2009. Oscar Rojo thanks the
support of Project Fondecyt Regular 1130135. Luis Medina and Oscar Rojo thanks the hospitality of the Mathematics
Department of Aveiro University, Aveiro, Portugal.\\




\begin{thebibliography}{9}


\bibitem{BHN10}
A. Brandst\"adt, C. Hundt, R. Nevries, Efficient edge domination on hole-free graphs in polynomial time.
LATIN 2010: theoretical informatics, 650–-661, \emph{Lecture Notes in Comput. Sci.}, 6034, Springer, Berlin, 2010.

\bibitem{BM11}
A. Brandst\"adt, R. Mosca, Dominating Induced Matchings for $P_7$-Free Graphs in Linear Time, \emph{Algorithmica}, in press.

\bibitem{BLR12}
A. Brandst\"adt, B. Leitert, D. Rautenbach, Efficient Dominating and Edge Dominating Sets for Graphs and Hypergraphs,
\emph{Lecture Notes in Comput. Sci.}, 7676 (2012), 267--277

\bibitem{CCDS08}
D. M. Cardoso, J. O. Cerdeira, C. Delorme, P. C. Silva, Efficient edge domination in regular graphs,
\emph{Discrete Applied Mathematics} {\bf 156} (2008): 3060--3065.

\bibitem{CL09}
D. M. Cardoso, V. V. Lozin, Dominating induced matchings. in Lipshteyn, Marina (ed.) et al., Graph theory,
computational intelligence and thought. Essays dedicated to Martin Charles Golumbic on the occasion of his
60th birthday. Berlin: Springer. \emph{Lecture Notes in Comput. Sci.} 5420 (2009), 77--86.

\bibitem{CKL11}
D. M. Cardoso, N. Korpelainen, V. V. Lozin, On the complexity of the dominating induced matching problem in
herditary classes of graphs, \emph{Discrete Appl. Math.} 159 (2011), 521--531.

\bibitem{cardoso_et_al2013}
D. M. Cardoso, M. A. A. de Freitas, E. A. Martins, M. Robbiano, Spectra of graphs obtained by a generalization
of the join graph operation, \emph{Discrete Mathematics} 313 (2013), 733--741.


\bibitem{CRS10}
D. Cvetkovi\'c, P. Rowlinson, S. Simi\'c, An Introduction to the Theory of Graph Spectra, Cambridge University Press
2010, Cambridge.


\bibitem{K09}
N. Korpelainen, A polynomial-time algorithm for the dominating induced matching problem in the class of convex
graphs, \emph{Electron. Notes Discrete Math.} 32 (2009) 133--140.

\end{thebibliography}
\end{document}